\documentclass[11pt]{article}

\usepackage[T1]{fontenc}
\usepackage[utf8]{inputenc}
\usepackage{lmodern}
\usepackage{microtype}
\usepackage{amsmath,amssymb,amsthm,mathtools}
\usepackage[hidelinks]{hyperref}
\usepackage[margin=1.12in]{geometry}

\newtheorem{theorem}{Theorem}[section]
\newtheorem{lemma}[theorem]{Lemma}
\newtheorem{proposition}[theorem]{Proposition}
\newtheorem{corollary}[theorem]{Corollary}
\theoremstyle{remark}

\newcommand{\T}{\mathbb T}
\newcommand{\R}{\mathbb R}
\newcommand{\Z}{\mathbb Z}
\newcommand{\e}{\mathrm e}
\newcommand{\dd}{\,\mathrm d}
\newcommand{\av}[1]{\langle #1\rangle}
\newcommand{\norm}[2][]{\left\lVert #2\right\rVert_{#1}}
\newcommand{\abs}[1]{\left|#1\right|}
\newcommand{\ip}[2]{\left\langle #1,#2\right\rangle}

\newcommand{\AuthorName}{Chenye Zhang}

\hypersetup{
  pdftitle={A Counterexample to a Bourgain--Brezis Sum-Space Estimate},
  pdfauthor={\AuthorName},
  pdfsubject={A counterexample separating ambient and gradient-constrained Sobolev decompositions}
}

\title{A Counterexample to a Bourgain--Brezis Sum-Space Estimate}
\author{Chenye Zhang\\
  \small School of Mathematical Sciences, Peking University, Beijing 100871, China\\
  \small \texttt{z\_cy@pku.edu.cn}}
\date{September 4, 2026}

\begin{document}
\maketitle

\begin{abstract}
A question of Bourgain and Brezis, later restated as Open Problem~7.1 by Brezis, asks whether the $H^{1/2}+W^{1,1}$ norm of a mean-zero scalar function is controlled by the ambient $H^{-1/2}+L^1$ sum norm of its gradient. We show that this estimate fails in every dimension $N\geq2$. More precisely, for every integer $K\geq16$, we construct a real-valued, smooth, mean-zero function $\Phi_K$ on $\T^N$ such that, for absolute constants $c,C>0$ independent of $N$ and $K$,
$$
  \norm[H^{-1/2}(\T^N)+L^1(\T^N)]{\nabla\Phi_K}\leq C,
  \qquad
  \norm[H^{1/2}(\T^N)+W^{1,1}(\T^N)]{\Phi_K}\geq cK^{1/4}-C.
$$
Equivalently, the ratio of the gradient-constrained decomposition cost to the ambient decomposition cost is unbounded. The construction applies the Hodge projection to a smooth approximation of a normal vector measure on a codimension-one sheet. Its potential records lacunary partial sums at geometrically separated distances from the sheet. A scale-path lower bound, combined with an odd critical Hardy inequality, excludes uniformly bounded $H^{1/2}+W^{1,1}$ decompositions.
\end{abstract}



\section{Introduction}

\subsection{From endpoint divergence equations to phase estimates}

The point of departure is an endpoint phenomenon discovered by Bourgain and Brezis. Let $N\geq2$ and let $f\in L^N(\T^N)$ have mean zero. The canonical solution $Y_0=\nabla\Delta^{-1}f$ of $\operatorname{div}Y=f$ belongs to $W^{1,N}$ by Calder\'on--Zygmund theory, but the critical Sobolev embedding $W^{1,N}\hookrightarrow L^\infty$ is false. Nevertheless, Bourgain and Brezis proved that the freedom in this underdetermined equation can be used to choose a different solution satisfying
\begin{equation}
 \operatorname{div}Y=f,
 \qquad
 \norm[W^{1,N}(\T^N)]{Y}+\norm[L^\infty(\T^N)]{Y}\leq C_N\norm[L^N(\T^N)]{f}.
 \label{eq:BB-divergence}
\end{equation}
This is one of the basic Bourgain--Brezis estimates \cite{BB03}; its extensions to div--curl and Hodge systems were developed in \cite{BB04,BB07}, and the surrounding theory is surveyed in \cite{VanSchaftingen2014}. The result is nonlinear in an essential sense: one cannot obtain the bounded representative in \eqref{eq:BB-divergence} from a bounded linear right inverse of the divergence operator \cite{BB03}.

By duality, \eqref{eq:BB-divergence} is equivalent to an endpoint estimate for scalar functions. In the notation of \cite[Section~7]{Brezis2023},
\begin{equation}
 \norm[L^{N/(N-1)}(\T^N)]{\psi-\av{\psi}}
 \leq C_N\norm[W^{-1,N/(N-1)}(\T^N)+L^1(\T^N)]{\nabla\psi}.
 \label{eq:BB-dual}
\end{equation}
Since $W^{1,N}(\T^N)$ embeds continuously into $H^{1/2}(\T^N)$, one obtains the weaker Hilbertian estimate
\begin{equation}
 \norm[L^{N/(N-1)}(\T^N)]{\psi-\av{\psi}}
 \leq C_N\norm[H^{-1/2}(\T^N)+L^1(\T^N)]{\nabla\psi}.
 \label{eq:BB-dual-Half}
\end{equation}

The original Bourgain--Brezis paper used this mechanism to control phases of circle-valued maps. If $\varphi$ is real valued and $u=\e^{i\varphi}$, then $\nabla\varphi=-i\overline u\nabla u$. Testing against the bounded solution $Y$ and using the fractional product estimate gives
\begin{equation}
 \norm[L^{N/(N-1)}(\T^N)]{\varphi-\av{\varphi}}
 \leq C_N\bigl(\lvert u\rvert_{H^{1/2}}+\lvert u\rvert_{H^{1/2}}^2\bigr).
 \label{eq:phase-Lp}
\end{equation}
A stronger phase estimate is available. Bourgain--Brezis \cite[Theorem~4]{BB03} and Bourgain--Brezis--Mironescu \cite[Theorem~3]{BBM04} proved the sum-space bound
\begin{equation}
 \norm[H^{1/2}(\T^N)+W^{1,1}(\T^N)]{\varphi-\av{\varphi}}
 \leq C_N\bigl(\lvert u\rvert_{H^{1/2}}+\lvert u\rvert_{H^{1/2}}^2\bigr).
 \label{eq:phase-strong}
\end{equation}
This estimate belongs to the broader critical lifting theory for Sobolev maps into the circle; see \cite{BBM00,BBM04,BM21}. The target on the left of \eqref{eq:phase-strong} is natural because both $H^{1/2}$ and $W^{1,1}$ embed into $L^{N/(N-1)}$ on the torus, while the sum norm retains substantially more information than the Lebesgue norm in \eqref{eq:phase-Lp}.

\subsection{Brezis' question and the main result}

The estimate at issue was isolated by Bourgain and Brezis in \cite[(1.16)]{BB03}. After proving \eqref{eq:phase-strong} by a direct paraproduct argument, they asked whether it could instead be recovered from a divergence equation by the duality mechanism leading to \eqref{eq:phase-Lp}. Brezis later restated the question as equation~(7.6) and Open Problem~7.1 in \cite{Brezis2023}. Thus, on $\T^N=(\R/2\pi\Z)^N$, the problem is whether there exists $C_N$ such that every smooth real-valued function $\varphi$ satisfies
\begin{equation*}
 \norm[H^{1/2}(\T^N)+W^{1,1}(\T^N)]{\varphi-\av{\varphi}}
 \leq C_N
 \norm[H^{-1/2}(\T^N)+L^1(\T^N)]{\nabla\varphi}.
 \tag{OP}\label{eq:brezisOP}
\end{equation*}
Here and below $\av{\varphi}$ denotes the normalized integral of $\varphi$ over the torus. For compatible Banach spaces $X$ and $Y$, the sum norm is
$$
 \norm[X+Y]{F}
 :=\inf_{F=F_X+F_Y}
       \bigl(\norm[X]{F_X}+\norm[Y]{F_Y}\bigr).
$$
We use the inhomogeneous norms throughout; in particular,
$$
 \norm[W^{1,1}]{w}=\norm[L^1]{w}+\norm[L^1]{\nabla w}.
$$
The notation $\lvert u\rvert_{H^{1/2}}$ in \eqref{eq:phase-Lp} and \eqref{eq:phase-strong} denotes the homogeneous seminorm. For real-valued data, the scalar sum norms may be computed over real-valued decompositions: taking real parts is contractive in all the spaces involved. Thus \eqref{eq:brezisOP} is not an ordinary Sobolev embedding. Its right-hand side permits an arbitrary vector decomposition $\nabla\varphi=F_0+F_1$, with $F_0\in H^{-1/2}$ and $F_1\in L^1$. Its left-hand side asks for the more rigid decomposition $\varphi=h+w$, which after differentiation forces both $\nabla h$ and $\nabla w$ to be gradients. The problem is whether this gradient constraint can be imposed at a uniformly comparable cost.

\begin{theorem}[Negative answer to Open Problem 7.1]\label{thm:main}
There are constants $c,C>0$, independent of $N$ and $K$, such that for every $N\geq2$ and every integer $K\geq16$ there is a function $\Phi_K\in C^\infty(\T^N;\R)$ with $\av{\Phi_K}=0$ satisfying
\begin{align}
 \norm[H^{-1/2}+L^1]{\nabla\Phi_K}&\leq C,
 \label{eq:rhs-main}\\
 \norm[H^{1/2}+W^{1,1}]{\Phi_K}&\geq cK^{1/4}-C.
 \label{eq:lhs-main}
\end{align}
In particular, \eqref{eq:brezisOP} is false in every dimension $N\geq2$.
\end{theorem}

For a compatible couple $(X_0,X_1)$, write
$$
 K(t,F;X_0,X_1)
 =\inf_{F=F_0+F_1}
   \bigl(\norm[X_0]{F_0}+t\norm[X_1]{F_1}\bigr).
$$
Let $\mathcal G_0\subset H^{-1/2}(\T^N;\R^N)$ and $\mathcal G_1\subset L^1(\T^N;\R^N)$ be, respectively, the gradients of mean-zero $H^{1/2}$ and $W^{1,1}$ functions, equipped with the norms inherited from the ambient spaces.

\begin{lemma}[Scalar and constrained decomposition costs]\label{lem:scalar-gradient-K}
For every mean-zero scalar distribution $\varphi$ on $\T^N$, with both sides understood as extended-valued quantities in $[0,\infty]$,
\begin{equation}
 \norm[H^{1/2}+W^{1,1}]{\varphi}
 \simeq_N K(1,\nabla\varphi;\mathcal G_0,\mathcal G_1).
 \label{eq:scalar-gradient-K}
\end{equation}
\end{lemma}

\begin{proof}
If $\varphi=h+w$ with $h\in H^{1/2}$ and $w\in W^{1,1}$, subtracting the respective means preserves the identity because $\av{h}+\av{w}=\av{\varphi}=0$. The resulting functions $h_0,w_0$ have mean zero and satisfy
$$
 \norm[H^{1/2}]{h_0}\leq\norm[H^{1/2}]{h},
 \qquad
 \norm[W^{1,1}]{w_0}\lesssim_N\norm[W^{1,1}]{w}.
$$
Their gradients give a constrained decomposition of $\nabla\varphi$. Conversely, if
$\nabla\varphi=\nabla h+\nabla w$ with $h,w$ mean zero, then
$\varphi=h+w$: the difference is constant and has mean zero. Moreover,
$$
 \norm[H^{1/2}]{h}\simeq\norm[H^{-1/2}]{\nabla h},
 \qquad
 \norm[W^{1,1}]{w}\lesssim_N\norm[L^1]{\nabla w}
$$
by the Fourier characterization and the mean-zero $L^1$ Poincar\'e inequality. Taking the two infima proves \eqref{eq:scalar-gradient-K}.
\end{proof}

\begin{corollary}[Ambient-versus-constrained decomposition gap]\label{cor:K-gap}
For $F_K=\nabla\Phi_K$ from Theorem~\ref{thm:main},
\begin{align*}
 K(1,F_K;H^{-1/2},L^1)&\leq C_N,\\
 K(1,F_K;\mathcal G_0,\mathcal G_1)&\geq c_NK^{1/4}-C_N.
\end{align*}
In particular,
$$
 \frac{K(1,F_K;\mathcal G_0,\mathcal G_1)}
      {K(1,F_K;H^{-1/2},L^1)}
 \xrightarrow[K\to\infty]{}\infty.
$$
\end{corollary}

\begin{proof}
The first estimate is \eqref{eq:rhs-main}; the second follows from \eqref{eq:lhs-main} and Lemma~\ref{lem:scalar-gradient-K}.
\end{proof}

Thus the obstruction persists for smooth data in every dimension $N\geq2$, and the gap arises when both summands are required to lie in the gradient subspaces. The construction yields the exponent $1/4$; its optimality is not addressed. Theorem~\ref{thm:main} is compatible with the nonlinear phase estimate \eqref{eq:phase-strong}, since the construction provides no uniform bound for $\lvert\e^{i\Phi_K}\rvert_{H^{1/2}}$. Consequently, the theorem rules out this linear duality route to \eqref{eq:phase-strong}, rather than the phase estimate itself.

\subsection{Position relative to nearby results}

The dimension-dependent critical Sobolev scale has a separate positive theory. Maz'ya proved a Hilbertian bounded-representative theorem at regularity $H^{N/2}$ \cite{Mazya2007}. In dimension two, Mironescu placed the Maz'ya and Bourgain--Brezis estimates in a common framework and discussed higher-dimensional analogues \cite{Mironescu2010}. Bousquet--Mironescu--Russ considered $\operatorname{div}Y=f$ with $f\in W^{s,p}$, $(s+1)p=N$, and $p\geq2$, and constructed solutions in $W^{s+1,p}\cap L^\infty$; they also treated Triebel--Lizorkin scales \cite{BMR13}. Bousquet--Russ--Wang--Yung extended the bounded-selection principle to critical higher-order and fractional Hodge systems in $\dot W^{s,p}$, $sp=N$, for a range of form degrees \cite{BRWY19}.

For the statements below, $L^1_0(\R^N;\Lambda^\ell)$ denotes the coefficientwise mean-zero subspace of $L^1(\R^N;\Lambda^\ell)$, and $d[X]$ denotes the image under the exterior derivative of a space $X$ of differential forms, equipped with the induced quotient norm
$$
 \norm[d[X]]{\eta}:=\inf\bigl\{\norm[X]{u}:u\in X,\ du=\eta\bigr\}.
$$
In recent preprints on $\R^N$, Ars\'enio proved
$$
 d[\dot H^{N/2}\Lambda^\ell]
 =d[(\dot H^{N/2}\cap L^\infty)\Lambda^\ell],
 \qquad 1\leq\ell\leq N-1,
$$
with equivalent induced norms \cite[Theorem~1.3]{Arsenio2026}, while Dai--Deng--Shi--Wu--Yang proved the corresponding $\dot W^{N/p,p}$ selection theorem for every $1<p<\infty$ and every $1\leq\ell\leq N-1$ \cite[Theorem~1.1]{DaiEtAl2026}. At the Hilbert point $p=2$, both the primal regularity $\dot H^{N/2}$ and the dual sum scale $\dot H^{-N/2}+L^1$ depend on the dimension. More precisely, for $\omega\in L^1_0(\R^N;\Lambda^\ell)$, Ars\'enio's dual estimate is
$$
 \norm[\dot H^{-N/2}(\R^N;\Lambda^\ell)]{\omega}
 \lesssim \norm[L^1(\R^N;\Lambda^\ell)]{\omega}
          +\norm[\dot H^{-N/2}(\R^N;\Lambda^\ell)]{P\omega}.
$$
By contrast, \eqref{eq:brezisOP} is an inhomogeneous torus estimate for a single exact vector field in $H^{-1/2}+L^1$, with exponent $-1/2$ independent of $N$.

Da Lio, Rivi\`ere and Wettstein proved a one-dimensional fractional Bourgain--Brezis estimate at the exponent $-1/2$ \cite[Theorem~3]{DLRW2021}. On $S^1$, their theorem gives
$$
 \norm[L^2(S^1)]{u-\av{u}}
 \lesssim
 \norm[\dot H^{-1/2}(S^1)+L^1(S^1)]{(-\Delta)^{1/4}u}
 +\norm[\dot H^{-1/2}(S^1)+L^1(S^1)]{\mathcal R(-\Delta)^{1/4}u},
$$
and they identify it with a characterization of the Bergman space $\mathcal A^2(\mathbb D)$ through uniform radial $L^1+\dot H^{-1/2}$ boundary norms. Their higher-dimensional theorem on $\T^N$ controls the full Riesz-conjugate family of $(-\Delta)^{N/4}u$ in $L^1+\dot H^{-N/2}$ and again concludes an $L^2$ estimate \cite[Theorem~4]{DLRW2021}. Thus their one-dimensional result differs from \eqref{eq:brezisOP} in the controlled object and in the target: it assumes bounds for a conjugate pair and concludes $L^2$ control, whereas \eqref{eq:brezisOP} assumes one ambient decomposition of the exact vector field $\nabla\varphi$ and asks for a scalar $H^{1/2}+W^{1,1}$ decomposition. Their higher-dimensional result also uses the dimension-dependent exponent $-N/2$.

Equation~\eqref{eq:brezisOP} compares the constrained and ambient $K$-functionals at the single parameter $t=1$. Full $K$-closedness requires a uniform comparison for every $t>0$ \cite{BerghLofstrom1976}. In another recent preprint, Curc\u{a} asks for such a full-parameter comparison for the different torus couple $(G(L^1),G(H^{-N/2}))$ \cite[Section~3.1, (35)]{Curca2026}. Curc\u{a}'s earlier work \cite{Curca2025} characterizes scalar Sobolev sum embeddings and constructs counterexamples by trace and fractal-set methods. The scale-path obstruction below instead concerns decompositions whose summands are constrained to lie in the gradient subspaces, and is tailored to the profile generated by the Hodge projection of a normal sheet.

\paragraph{Novelty of the construction.}
The proof combines standard Hodge projections and periodic layer potentials, classical fourth-moment estimates for lacunary trigonometric series \cite{Zygmund2002}, and an endpoint odd Hardy inequality related to $H^{1/2}_{00}$ \cite{LionsMagenes1972}. The new ingredient is the normal-sheet-to-scale-path reduction: Proposition~\ref{prop:profile} encodes lacunary partial sums in normal slices, while Proposition~\ref{prop:compression} converts every scalar decomposition into a discrete path with controlled $L^1$ variation and square-summable $L^2$ tracking error. Lemma~\ref{lem:path} then yields the unbounded constrained-to-ambient gap.

Section~\ref{sec:heuristic} explains the mechanism, Sections~\ref{sec:construction}--\ref{sec:two-dimensional-proof} prove the two-dimensional estimates, and Section~\ref{sec:higher-dimensions} extends them to $N>2$.

\section{The mechanism behind the counterexample}
\label{sec:heuristic}

\subsection{Creating a gap with the Hodge projection}

The decomposition asymmetry in \eqref{eq:brezisOP} suggests beginning on the vector-field side. Let $P$ be the $L^2$ Hodge projection onto gradient fields and let $Q=I-P$. For any vector field $a$ for which the following terms make sense, the gradient field $F=2Pa$ has the ambient decomposition
\begin{equation}
 F=2Pa=2a-2Qa.
 \label{eq:heuristic-hodge}
\end{equation}
Consequently $F$ has small $H^{-1/2}+L^1$ norm whenever $a$ has small $L^1$ norm and $Qa$ has small $H^{-1/2}$ norm. Neither $a$ nor $Qa$ is generally a gradient, so \eqref{eq:heuristic-hodge} need not be a constrained decomposition. We choose $a$ so that every scalar $H^{1/2}+W^{1,1}$ decomposition of the potential associated with $Pa$ has large cost.

Why should $a$ concentrate near a hypersurface? Work provisionally on $\T^2$ with tangential coordinate $x$ and normal coordinate $t$, and ignore normalization constants. For a mean-zero tangential datum $b$, consider the ideal normal sheet
$$
 a=b(x)e_t\delta_{\{t=0\}},
$$
At frequency $(n,\ell)$, the Hodge multiplier satisfies
$$
 \abs{Q(n,\ell)e_t}^2=\frac{n^2}{n^2+\ell^2}.
$$
Thus the normal projection supplies exactly the cancellation missing from the bare sheet:
\begin{equation}
 \norm[H^{-1/2}]{Qa}^2
 \lesssim \sum_{n\ne0}\abs{\widehat b(n)}^2\sum_{\ell\in\Z}\frac{n^2}{(n^2+\ell^2)^{3/2}}
 \lesssim \norm[L^2]{b}^2.
 \label{eq:heuristic-sheet-cancellation}
\end{equation}
Without the factor $n^2/(n^2+\ell^2)$, the normal-frequency sum is logarithmically divergent. This suggests a normal codimension-one concentration: it is compatible with the $H^{-1/2}$ endpoint after the coexact Hodge component is removed. Although the ideal sheet is a measure rather than an $L^1$ function, its normal mollifications have uniformly bounded $L^1$ norm; this is the regularization used in the proof.

\subsection{Turning normal distance into a frequency parameter}

Let $\Phi=2\Delta^{-1}\operatorname{div}a$, so that $\nabla\Phi=2Pa$. The potential generated by one tangential mode $\cos(nx)$ has, on the positive side of the sheet, the form
$$
 A_n(t)\cos(nx),
 \qquad
 A_n(t)=\frac{\sinh(n(\pi-t))}{\sinh(n\pi)}.
$$
For $n\geq1$,
$$
 A_n(t)=\e^{-nt}
 \frac{1-\e^{-2n(\pi-t)}}{1-\e^{-2n\pi}}
 \asymp_a \e^{-nt},
 \qquad 0<t\leq a<\pi.
$$
In particular $A_n(0+)=1$, while the mode has essentially disappeared once $t$ is appreciably larger than $n^{-1}$. Normal distance therefore acts as an inverse frequency cutoff.

This observation suggests choosing a lacunary tangential datum
$$
 b_K(x)=K^{-1/2}\sum_{i=1}^K\cos(\lambda_i x),
 \qquad \lambda_i=R^{M+i}.
$$
The factor $K^{-1/2}$ keeps $\norm[L^2]{b_K}$, and hence both costs in \eqref{eq:heuristic-hodge}, uniformly bounded. At distance $t\simeq\lambda_j^{-1}$, the modes with $i\ll j$ remain close to their boundary values, only a bounded number of transition modes have an error of order one, and the modes with $i>j$ lie beyond their natural decay scale. A tangential low-pass multiplier $S_j$ separates these regimes sharply. The averaged potential on the slab $I_j=[\lambda_j^{-1},2\lambda_j^{-1}]$ then satisfies, in a square-summed sense made precise in Proposition~\ref{prop:profile},
$$
 S_j\frac1{|I_j|}\int_{I_j}\Phi_K(\cdot,t)\dd t\approx m_j,
 \qquad
 m_j(x)=K^{-1/2}\sum_{i=1}^j\cos(\lambda_i x).
$$
The potential has encoded the ordered list of lacunary partial sums as one moves toward the sheet.

\subsection{The obstruction to scalar decompositions}

Suppose $\Phi_K=w+h$ with $w\in W^{1,1}$ and $h\in H^{1/2}$. Average on the slabs and set $f_j=S_j|I_j|^{-1}\int_{I_j}w(\cdot,t)\dd t$. The $t$-variation of $w$, together with the first-moment estimate for the change from $S_j$ to $S_{j+1}$, controls the total path variation
$$
 A:=\sum_{j=1}^{K-1}\norm[L^1]{f_{j+1}-f_j}
 \lesssim \norm[L^1]{\nabla w}.
$$
The sheet potential is odd in $t$. After taking odd parts, the critical Hardy inequality controls the square-summed slab averages of $h$ and hence the tracking error
$$
 B:=\left(\sum_{j=1}^K\norm[L^2]{m_j-f_j}^2\right)^{1/2}
 \lesssim 1+\norm[H^{1/2}]{h}.
$$
Oddness is essential here: at the borderline exponent $H^{1/2}$ there is no general point trace, but oddness supplies the cancellation at the sheet that makes the weighted integral $\int_0^a\norm[L^2_x]{h(\cdot,t)}^2\dd t/t$ finite.

It remains to see why the path $m_j$ cannot simultaneously have small variation and be followed with small square error. Lacunarity gives
$$
 \norm[L^1]{m_{j+L}-m_j}\gtrsim\sqrt{\frac{L}{K}}.
$$
Comparing $m_{j+L}-m_j$ with $f_{j+L}-f_j$ and summing in $j$ yields the schematic inequality
$$
 \sqrt{KL}\lesssim LA+\sqrt K\,B.
$$
The choice $L\simeq\sqrt K$ gives $A+B\gtrsim K^{1/4}$ and disproves \eqref{eq:brezisOP}. Finally, replacing the ideal sheet by a mollification at width $\varepsilon_K\ll\lambda_K^{-1}$ makes $\Phi_K$ smooth without changing any of the sampled profiles by more than a uniformly square-summable error.

\section{Hodge reformulation and the counterexample}
\label{sec:construction}

\subsection{Normalizations and the Hodge projection}

We use normalized Haar measure
$$
 dm_d(x)=(2\pi)^{-d}\dd x
$$
on $\T^d$.  All $L^p$ norms are taken with respect to $dm_d$, and
\begin{equation}
 \widehat f(k)=\int_{\T^d}f(x)\e^{-ik\cdot x}\,dm_d(x),
 \qquad
 f(x)=\sum_{k\in\Z^d}\widehat f(k)\e^{ik\cdot x}.
 \label{eq:fourier-convention}
\end{equation}
Functions are embedded into distributions by integration against $dm_d$. Thus the Haar-normalized Dirac mass $\delta_0$, defined by $\langle\delta_0,\psi\rangle=\psi(0)$, has Fourier coefficients equal to one.  Convolution is normalized by
\begin{equation}
 (f*g)(x)=\int_{\T^d}f(x-y)g(y)\,dm_d(y),
 \qquad \widehat{f*g}(k)=\widehat f(k)\widehat g(k).
 \label{eq:convolution-convention}
\end{equation}
We write $*_t$ for convolution only in the last variable.  Interval averages use coordinate length: $|I|^{-1}\int_I f(t)\dd t$.

For $s\in\R$, set
$$
 \norm[H^s(\T^d)]{f}^2
  =\sum_{\xi\in\Z^d}(1+\abs{\xi}^2)^s
     \abs{\widehat f(\xi)}^2.
$$
Vector-valued norms use the Euclidean norm. In particular, for $F:\T^d\to\R^d$,
$$
 \norm[L^1(\T^d;\R^d)]{F}
 =\int_{\T^d}\abs{F(x)}_{\ell^2}\,dm_d(x),
 \qquad
 \norm[H^s(\T^d;\R^d)]{F}^2
 =\sum_{\xi\in\Z^d}(1+\abs{\xi}^2)^s
   \abs{\widehat F(\xi)}_{\ell^2}^2.
$$
All implicit constants below are absolute unless a dependence is displayed explicitly.

Let $P$ be the Hodge projection onto gradient fields and $Q=I-P$.  At nonzero frequencies,
\begin{equation}
 \widehat{PF}(\xi)=\frac{\xi\otimes\xi}{\abs{\xi}^2}\widehat F(\xi),
 \qquad
 \widehat{QF}(\xi)
 =\left(I-\frac{\xi\otimes\xi}{\abs{\xi}^2}\right)\widehat F(\xi).
 \label{eq:hodge-symbol}
\end{equation}
At zero frequency we set
$$
 \widehat{PF}(0)=0,
 \qquad
 \widehat{QF}(0)=\widehat F(0).
$$
For a mean-zero scalar distribution $\phi$, $\nabla\phi\in H^{-1/2}$ if and only if $\phi\in H^{1/2}$, with equivalent norms. Likewise, if $\nabla\phi\in L^1$ distributionally, then $\phi\in W^{1,1}$ and $\norm[W^{1,1}]{\phi}\lesssim\norm[L^1]{\nabla\phi}$. To justify the latter statement without assuming a priori that $\phi$ is a function, convolve $\phi$ with a smooth approximate identity and apply the mean-zero $L^1$ Poincar\'e inequality. The regularizations are uniformly bounded in $BV$; compactness gives an $L^1$-convergent subsequence, and distributional convergence identifies its limit with $\phi$, whose weak derivatives are the prescribed $L^1$ functions.

Together with Lemma~\ref{lem:scalar-gradient-K}, these observations identify \eqref{eq:brezisOP}, up to Poincar\'e constants, with the comparison of constrained and ambient costs in Corollary~\ref{cor:K-gap}.

\subsection{Definition of the sequence}

Fix
$$
 R=4,
 \qquad a=\frac{\pi}{8},
 \qquad M=1.
$$
Then $3R^{-(M+1)}<a$.
For every integer $K\geq16$, put
$$
 \lambda_i=R^{M+i},
 \qquad
 b_K(x)=K^{-1/2}\sum_{i=1}^K\cos(\lambda_i x).
$$
In particular, every $\lambda_i$ is an integer. Work first on $\T^2$ with coordinates $(x,t)$. Let $e_t=(0,1)$. Choose an even function $\eta\in C_c^\infty((-1,1))$ with $\eta\geq0$ and $\int_\R\eta(s)\dd s=1$, and let $\rho_\varepsilon$ be the $2\pi$-periodization of
$$
 \frac{2\pi}{\varepsilon}\eta(t/\varepsilon).
$$
Then $\rho_\varepsilon$ is supported in $(-\varepsilon,\varepsilon)$ modulo $2\pi$ and $\int_\T\rho_\varepsilon\,dm_1=1$. Define
\begin{equation}
 \varepsilon_K=K^{-2}\lambda_K^{-1},
 \qquad
 a_K(x,t)=\frac1{2\pi}b_K(x)e_t\rho_{\varepsilon_K}(t),
 \label{eq:aK}
\end{equation}
and, using the mean-zero inverse of the Laplacian,
\begin{equation}
 \Phi_K=2\Delta^{-1}\operatorname{div}a_K.
 \label{eq:PhiK}
\end{equation}
For $n\neq0$ and $\ell\in\Z$, the definitions give
\begin{align}
 \widehat{a_K}(n,\ell)
 &=\frac1{2\pi}\widehat b_K(n)
   \widehat{\rho_{\varepsilon_K}}(\ell)e_t,
 \label{eq:aK-fourier}\\
 \widehat{\Phi_K}(n,\ell)
 &=-\frac{2i\ell}{2\pi(n^2+\ell^2)}
   \widehat b_K(n)
   \widehat{\rho_{\varepsilon_K}}(\ell).
 \label{eq:PhiK-fourier}
\end{align}
Thus $\Phi_K$ is smooth, real-valued, mean-zero, and odd in $t$. Direct multiplication by $i(n,\ell)$ yields
\begin{equation}
 \nabla\Phi_K=2Pa_K=2a_K-2Qa_K.
 \label{eq:ambient-decomp-preview}
\end{equation}
This identity exhibits the ambient decomposition used for the upper bound. The proof of Theorem \ref{thm:main} is reduced to the two estimates
\begin{equation}
 \norm[H^{-1/2}+L^1]{\nabla\Phi_K}\lesssim1,
 \qquad
 \norm[H^{1/2}+W^{1,1}]{\Phi_K}\gtrsim K^{1/4}-1.
 \label{eq:two-targets}
\end{equation}
The first is a direct Fourier estimate for $Qa_K$; the second occupies the next two sections.

\section{The scale-path lower bound}
\label{sec:scale-path}

For $1\leq j\leq K$, define
\begin{equation}
 m_j(x)=K^{-1/2}\sum_{i=1}^j\cos(\lambda_i x).
 \label{eq:mj}
\end{equation}

\subsection{Lacunary blocks and paths}

The quasi-independence and moment bounds of lacunary trigonometric systems are classical; see, for example, \cite[Chapter~V]{Zygmund2002}. We include the short fourth-moment argument needed here.

\begin{lemma}[Lacunary block lower bound]\label{lem:lacunary-l1}
There is an absolute constant $c>0$ such that, whenever $1\leq L\leq K-j$,
\begin{equation}
 \norm[L^1(\T)]{m_{j+L}-m_j}
 \geq c\sqrt{\frac{L}{K}}.
 \label{eq:lacunary-block}
\end{equation}
\end{lemma}

\begin{proof}
Put $g=\sum_{i=j+1}^{j+L}\cos(\lambda_i x)$.  Orthogonality gives $\norm[L^2]{g}^2=L/2$.  We claim that
\begin{equation}
 \norm[L^4]{g}^4\leq CL^2.
 \label{eq:l4-lacunary}
\end{equation}
After expanding into exponentials, a constant term corresponds to a relation
$$
 \varepsilon_1\lambda_{i_1}+\cdots+
 \varepsilon_4\lambda_{i_4}=0,
 \qquad \varepsilon_\nu\in\{-1,1\}.
$$
Let $m=\max_\nu i_\nu$ and cancel every opposite-sign pair among the occurrences of $\lambda_m$. If an occurrence remains, its net contribution has absolute value at least $\lambda_m$, while the at most three remaining lower-frequency terms have total absolute value at most $3\lambda_{m-1}=3\lambda_m/R<\lambda_m$, a contradiction. All occurrences of the maximal frequency therefore cancel in opposite-sign pairs. Repeating the argument at the next remaining frequency shows that every relation is a union of two opposite-sign equal-frequency pairs, including relations with repeated maximal indices. There are three pairings of the four positions; after a pairing is fixed, the two frequency indices and the two signs can be chosen in at most $4L^2$ ways. Thus the number of ordered signed relations is at most $12L^2$, which proves \eqref{eq:l4-lacunary}. Interpolation yields
$$
 \norm[L^2]{g}
 \leq \norm[L^1]{g}^{1/3}\norm[L^4]{g}^{2/3},
 \qquad
 \norm[L^1]{g}\geq
 \frac{\norm[L^2]{g}^3}{\norm[L^4]{g}^2}\gtrsim\sqrt L.
$$
Division by $\sqrt K$ proves the lemma.
\end{proof}

\begin{lemma}[Scale-path obstruction]\label{lem:path}
For arbitrary $f_1,\ldots,f_K\in L^2(\T)$ set
$$
 A=\sum_{j=1}^{K-1}\norm[L^1]{f_{j+1}-f_j},
 \qquad
 B=\left(\sum_{j=1}^K\norm[L^2]{m_j-f_j}^2\right)^{1/2}.
$$
If $K\geq16$, then
\begin{equation}
 A+B\geq cK^{1/4}.
 \label{eq:path-obstruction}
\end{equation}
\end{lemma}

\begin{proof}
Let $e_j=m_j-f_j$ and choose $1\leq L\leq K/2$.  For $1\leq j\leq K-L$, Lemma \ref{lem:lacunary-l1} and the triangle inequality give
$$
 c\sqrt{L/K}
 \leq \norm[L^1]{e_{j+L}}+\norm[L^1]{e_j}
       +\sum_{r=j}^{j+L-1}\norm[L^1]{f_{r+1}-f_r}.
$$
After summing in $j$, every nearest-neighbor increment occurs at most $L$ times.  Since $dm_1(\T)=1$,
$$
 c\sqrt{KL}
 \leq LA+2\sum_{j=1}^K\norm[L^1]{e_j}
 \leq LA+2\sqrt K B.
$$
Taking $L=\lfloor\sqrt K\rfloor$ and dividing by $\sqrt K$ proves \eqref{eq:path-obstruction}.
\end{proof}

\subsection{An odd critical Hardy inequality}

The weight $t^{-1}$ is the endpoint zero-boundary phenomenon associated with the Lions--Magenes space $H^{1/2}_{00}$ \cite{LionsMagenes1972}. The odd extension supplies the required cancellation at $t=0$. We record the periodic Hilbert-valued form used below.

\begin{lemma}[Hilbert-valued odd Hardy inequality]\label{lem:hardy}
Let $\mathcal H$ be a real or complex Hilbert space, let $0<a<\pi/4$, and let $g\in H^{1/2}(\T;\mathcal H)$ be odd. There is an absolute constant $C$ such that
\begin{equation}
 \int_0^a\frac{\norm[\mathcal H]{g(t)}^2}{t}\dd t
 \leq C\norm[H^{1/2}(\T;\mathcal H)]{g}^2.
 \label{eq:odd-hardy}
\end{equation}
Consequently, if $h\in H^{1/2}(\T^2)$ is odd in $t$, then
\begin{equation}
 \int_0^a\frac{\norm[L^2_x]{h(\cdot,t)}^2}{t}\dd t
 \leq C\norm[H^{1/2}(\T^2)]{h}^2.
 \label{eq:odd-hardy-2d}
\end{equation}
\end{lemma}

\begin{proof}
The periodic analogue of the standard scalar Gagliardo characterization \cite{DNPV2012} gives
$$
 \lvert g\rvert_{H^{1/2}(\T;\mathcal H)}^2
 \asymp
 \int_\T\int_\T
 \frac{\norm[\mathcal H]{g(r)-g(s)}^2}
      {d_\T(r,s)^2}\dd r\dd s,
$$
where $d_\T$ is the geodesic distance. The Hilbert-valued version follows by applying the scalar periodic identity to the coordinates in an orthonormal basis and using Tonelli's theorem. First suppose that $g$ is smooth. If $0<t,s<a<\pi/4$, then $d_\T(t,-s)=t+s$. Oddness gives $g(-s)=-g(s)$, so restricting the preceding double integral to the pairs $(t,-s)$ gives
\begin{equation}
 \norm[H^{1/2}]{g}^2\gtrsim
 \int_0^a\int_0^a
 \frac{\norm[\mathcal H]{g(t)+g(s)}^2}{(t+s)^2}\dd t\dd s.
 \label{eq:cross-gagliardo}
\end{equation}
The mixed quadratic form is nonnegative.  Indeed,
$$
 \frac1{(t+s)^2}=\int_0^\infty r\e^{-r(t+s)}\dd r
$$
and hence
$$
 \operatorname{Re}\int_0^a\int_0^a
 \frac{\ip{g(t)}{g(s)}_{\mathcal H}}{(t+s)^2}\dd t\dd s
 =\int_0^\infty r
   \norm[\mathcal H]{\int_0^a\e^{-rt}g(t)\dd t}^2\dd r\geq0.
$$
The two diagonal terms in \eqref{eq:cross-gagliardo} equal
$$
 2\int_0^a\norm[\mathcal H]{g(t)}^2
   \left(\frac1t-\frac1{t+a}\right)\dd t
 \geq \int_0^a\frac{\norm[\mathcal H]{g(t)}^2}{t}\dd t.
$$
This proves \eqref{eq:odd-hardy} for smooth odd functions.  For general $g$, take symmetric Fourier truncations $g_N$.  They remain odd and converge to $g$ in $H^{1/2}(\T;\mathcal H)$.  A subsequence converges almost everywhere, after restricting to the separable closed subspace generated by the Fourier coefficients of $g$.  Fatou's lemma gives
$$
 \int_0^a\frac{\norm[\mathcal H]{g(t)}^2}{t}\dd t
 \leq\liminf_N\int_0^a
       \frac{\norm[\mathcal H]{g_N(t)}^2}{t}\dd t
 \leq C\norm[H^{1/2}]{g}^2.
$$
Taking $\mathcal H=L^2(\T_x)$ proves \eqref{eq:odd-hardy-2d}, because $(1+\ell^2)^{1/2}\leq(1+n^2+\ell^2)^{1/2}$ makes the resulting $H_t^{1/2}(L_x^2)$ norm no larger than the full $H^{1/2}(\T^2)$ norm.
\end{proof}

\subsection{Compression of scalar decompositions}
\label{sec:compression}

Set
\begin{equation}
 I_j=[\lambda_j^{-1},2\lambda_j^{-1}],
 \qquad 1\leq j\leq K.
 \label{eq:slabs}
\end{equation}
The intervals are pairwise disjoint and, by the choice of $M$, lie in $(0,a)$. Choose an even real function $\chi\in C_c^\infty(\R)$ satisfying
$$
 0\leq\chi\leq1,
 \qquad \chi=1\ \hbox{on }[-1,1],
 \qquad \operatorname{supp}\chi\subset(-R/2,R/2),
$$
and define the tangential multiplier
$$
 \widehat{S_jF}(n)=\chi(n/\lambda_j)\widehat F(n).
$$
For $\lambda>0$ let
$$
 K_\lambda(x)=\sum_{n\in\Z}\chi(n/\lambda)\e^{inx}.
$$
This is a trigonometric polynomial, and $S_jF=K_{\lambda_j}*F$ defines $S_jF$ for every $F\in L^1(\T)$. In particular, $S_j$ is an $L^2$ contraction.

\begin{proposition}[Decomposition compression]\label{prop:compression}
Let $\Phi\in C^\infty(\T^2)$ be odd in $t$, and set
$$
 \Psi_j(x)=\frac1{|I_j|}\int_{I_j}\Phi(x,t)\dd t.
$$
Define
\begin{equation}
 E^2:=\sum_{j=1}^K\norm[L^2_x]{m_j-S_j\Psi_j}^2.
 \label{eq:profile-error}
\end{equation}
Then every decomposition $\Phi=w+h$ with $w\in W^{1,1}(\T^2)$ and $h\in H^{1/2}(\T^2)$ satisfies
\begin{equation}
 \norm[L^1]{\nabla w}+\norm[H^{1/2}]{h}
 \geq cK^{1/4}-C E.
 \label{eq:compression-lower}
\end{equation}
\end{proposition}

\begin{proof}
The odd-part projection
$$
 \mathcal Of(x,t)=\tfrac12\bigl(f(x,t)-f(x,-t)\bigr)
$$
is contractive on $W^{1,1}$ and $H^{1/2}$.  We may therefore assume that $w$ and $h$ are odd.  Define
$$
 W_j=\frac1{|I_j|}\int_{I_j}w(\cdot,t)\dd t,
 \qquad
 H_j=\frac1{|I_j|}\int_{I_j}h(\cdot,t)\dd t,
 \qquad f_j=S_jW_j.
$$
Fubini gives $W_j\in L^1(\T)$.  Since $\chi$ is compactly supported, $S_jW_j$ has finitely many Fourier modes.  Thus $f_j\in L^2(\T)$, as required in Lemma \ref{lem:path}.

We first prove
\begin{equation}
 \sum_{j=1}^{K-1}\norm[L^1_x]{f_{j+1}-f_j}
 \leq C\norm[L^1(\T^2)]{\nabla w}.
 \label{eq:variation-path}
\end{equation}
Write
$$
 f_{j+1}-f_j
 =S_{j+1}(W_{j+1}-W_j)+(S_{j+1}-S_j)W_j.
$$
With
$$
 \check\chi(y)=\frac1{2\pi}\int_\R\chi(\xi)\e^{iy\xi}\dd\xi,
$$
Poisson summation (with the conventions of Section \ref{sec:construction}; see also \cite[Chapter~VI, \S~1]{Katznelson2004}) gives
$$
 K_\lambda(x)=2\pi\lambda\sum_{q\in\Z}
   \check\chi\bigl(\lambda(x+2\pi q)\bigr).
$$
Writing $d_\T(x,0)=|x|$ on $[-\pi,\pi]$, a change of variables gives
\begin{align}
 \int_\T|K_\lambda(x)|\,dm_1(x)
 &\leq\int_\R|\check\chi(y)|\dd y=:C_0,
 \label{eq:kernel-l1}\\
 \int_\T d_\T(x,0)|K_\lambda(x)|\,dm_1(x)
 &\leq\lambda^{-1}\int_\R|y\check\chi(y)|\dd y
 =:C_1\lambda^{-1}.
 \label{eq:kernel-moment}
\end{align}
For $D_j=K_{\lambda_{j+1}}-K_{\lambda_j}$, $\int_\T D_j\,dm_1=0$ and
\begin{equation}
 \int_\T d_\T(x,0)|D_j(x)|\,dm_1(x)
 \leq C\lambda_j^{-1}.
 \label{eq:difference-kernel-moment}
\end{equation}
The first term in the difference of $f_{j+1}$ and $f_j$ is controlled by the uniform $L^1$ bound for $K_\lambda$. If $t\in I_{j+1}$ and $s\in I_j$, then $t<s$, and the fundamental theorem of calculus gives
$$
 \norm[L^1_x]{w(\cdot,t)-w(\cdot,s)}
 \leq\int_t^s\norm[L^1_x]{\partial_rw(\cdot,r)}\dd r.
$$
Averaging this inequality gives
\begin{align*}
 \norm[L^1_x]{W_{j+1}-W_j}
 &\leq\frac1{|I_j||I_{j+1}|}
   \int_{s\in I_j}\int_{t\in I_{j+1}}\int_t^s
   \norm[L^1_x]{\partial_rw(\cdot,r)}\dd r\dd t\dd s\\
 &\leq \int_{J_j}\norm[L^1_x]{\partial_rw(\cdot,r)}\dd r,
\end{align*}
where $J_j$ is the convex hull of $I_j\cup I_{j+1}$. Since $R=4$, every point belongs to at most two of the intervals $J_j$. For the second term, cancellation of $D_j$, the translation estimate $\norm[L^1]{\tau_yW_j-W_j}\leq d_\T(y,0)\norm[L^1]{\partial_xW_j}$, and \eqref{eq:difference-kernel-moment} give
\begin{align*}
 \norm[L^1_x]{(S_{j+1}-S_j)W_j}
 &\leq C\lambda_j^{-1}\norm[L^1_x]{\partial_xW_j}\\
 &\leq C\lambda_j^{-1}|I_j|^{-1}
       \int_{I_j}\norm[L^1_x]{\partial_xw(\cdot,t)}\dd t.
\end{align*}
Since $|I_j|=\lambda_j^{-1}$ and the $I_j$ are disjoint, summing proves \eqref{eq:variation-path}.

Jensen's inequality gives
$$
 \norm[L^2_x]{H_j}^2\leq |I_j|^{-1}\int_{I_j}\norm[L^2_x]{h(\cdot,t)}^2\dd t.
$$
Because $|I_j|^{-1}=\lambda_j\leq2/t$ on $I_j$ and the intervals $I_j$ are disjoint, Lemma \ref{lem:hardy} yields
\begin{equation}
 \sum_{j=1}^K\norm[L^2_x]{H_j}^2
 \leq2\int_0^a\frac{\norm[L^2_x]{h(\cdot,t)}^2}{t}\dd t
 \leq C\norm[H^{1/2}]{h}^2.
 \label{eq:average-hardy}
\end{equation}
Since $\Psi_j=W_j+H_j$ and $S_j$ is an $L^2$ contraction,
\begin{align}
 \left(\sum_{j=1}^K\norm[L^2_x]{m_j-f_j}^2\right)^{1/2}
 &\leq E+
   \left(\sum_{j=1}^K\norm[L^2_x]{S_jH_j}^2\right)^{1/2}\notag\\
 &\leq E+C\norm[H^{1/2}]{h}.
 \label{eq:path-B-bound}
\end{align}
Lemma \ref{lem:path}, together with \eqref{eq:variation-path} and \eqref{eq:path-B-bound}, proves \eqref{eq:compression-lower}.
\end{proof}

\section{The normal-sheet profile}
\label{sec:profile}

The next proposition identifies the lacunary path encoded by the construction of Section \ref{sec:construction}.

\begin{proposition}[Profile estimate]\label{prop:profile}
Let
$$
 \Psi_j(x)=\frac1{|I_j|}\int_{I_j}\Phi_K(x,t)\dd t.
$$
Then
\begin{equation}
 \sum_{j=1}^K\norm[L^2_x]{m_j-S_j\Psi_j}^2\leq C,
 \label{eq:profile-uniform}
\end{equation}
uniformly in $K$.
\end{proposition}

\begin{proof}
We first replace the normal mollifier by an ideal sheet.  Let $\mu_K=(2\pi)^{-1}b_K(x)e_t\delta_0$ and $\Phi_{K,0}=2\Delta^{-1}\operatorname{div}\mu_K$.  Under the convention $\Delta\e^{i(nx+\ell t)}=-(n^2+\ell^2)\e^{i(nx+\ell t)}$, the periodic Green function for $\partial_t^2-n^2$ is
\begin{equation}
 \mathcal G_n(t)
 =-\sum_{\ell\in\Z}\frac{\e^{i\ell t}}{n^2+\ell^2}
 =-\frac{\pi\cosh(n(\pi-|t|))}{n\sinh(n\pi)},
 \qquad (\partial_t^2-n^2)\mathcal G_n=\delta_0.
 \label{eq:normalized-green}
\end{equation}
The coordinate derivative of $\mathcal G_n$ has jump $2\pi$ at zero. Indeed, when functions are embedded using $dm_1=dt/(2\pi)$, a coordinate derivative jump $J$ contributes $(J/2\pi)\delta_0$.  Thus, for $0<t<\pi$,
\begin{equation}
 2\frac1{2\pi}\mathcal G_n'(t)
 =\frac{\sinh(n(\pi-t))}{\sinh(n\pi)}=:A_n(t).
 \label{eq:green-normalization}
\end{equation}
Let $A_n$ denote the odd, $2\pi$-periodic function whose restriction to $(0,\pi)$ is the right-hand side of \eqref{eq:green-normalization}. Consequently
\begin{equation}
 \Phi_{K,0}(x,t)=K^{-1/2}\sum_{i=1}^K
 A_{\lambda_i}(t)\cos(\lambda_i x).
 \label{eq:sheet-profile}
\end{equation}
In particular, $A_n(0+)=1$. Differentiation gives, for $0<t<\pi$ and $n\geq1$,
$$
 \abs{A_n'(t)}
 =n\frac{\cosh(n(\pi-t))}{\sinh(n\pi)}
 \leq n\coth(n\pi)\leq n\coth\pi.
$$
The mean value theorem therefore gives, uniformly for $0<t<\pi$ and $n\geq1$,
\begin{equation}
 \abs{1-A_n(t)}\leq Cnt.
 \label{eq:profile-Lipschitz}
\end{equation}

By \eqref{eq:convolution-convention}, $a_K=\mu_K*_t\rho_{\varepsilon_K}$. Hence $\Phi_K$ is the Haar convolution of $\Phi_{K,0}$ with $\rho_{\varepsilon_K}$ in the $t$ variable. Since $K\geq16$, one has $\varepsilon_K<\tfrac12\lambda_K^{-1}$. Thus, for $t\in I_j$ and for the representative $s\in(-\varepsilon_K,\varepsilon_K)$ of a point in $\operatorname{supp}\rho_{\varepsilon_K}$,
$$
 0<t-s<3R^{-(M+1)}<a.
$$
In particular, the convolution does not cross the sheet on any $I_j$, and \eqref{eq:profile-Lipschitz} applies throughout the convolution window. Define
$$
 \alpha_{ij}
 :=\frac1{|I_j|}\int_{I_j}
   (A_{\lambda_i}*_t\rho_{\varepsilon_K})(t)\dd t.
$$
Then, for $i\leq j$,
\begin{equation}
 \abs{1-\alpha_{ij}}
 \leq C\bigl(\lambda_i\lambda_j^{-1}
              +\varepsilon_K\lambda_i\bigr).
 \label{eq:alpha-error}
\end{equation}
The multiplier $S_j$ fixes $\cos(\lambda_i x)$ for $i\leq j$ and annihilates it for $i\geq j+1$.  Orthogonality therefore gives
\begin{align*}
 \sum_{j=1}^K\norm[L^2_x]{m_j-S_j\Psi_j}^2
 &=\frac1{2K}\sum_{j=1}^K\sum_{i=1}^j
     \abs{1-\alpha_{ij}}^2\\
 &\leq \frac{C}{K}\sum_{j=1}^K\sum_{i=1}^jR^{-2(j-i)}
   +\frac{C\varepsilon_K^2}{K}
       \sum_{i=1}^K(K-i+1)\lambda_i^2.
\end{align*}
The first term is bounded by a geometric series, while
$$
 \sum_{i=1}^K(K-i+1)\lambda_i^2\leq C\lambda_K^2
$$
and $\varepsilon_K\lambda_K=K^{-2}$.  This proves \eqref{eq:profile-uniform}.
\end{proof}

\section{Completion of the two-dimensional proof}
\label{sec:two-dimensional-proof}

\subsection{The scalar lower bound}

Apply Proposition \ref{prop:compression} to $\Phi_K$.  Proposition \ref{prop:profile} makes the error $E$ uniformly bounded, so every decomposition $\Phi_K=w+h$ satisfies
$$
 \norm[L^1]{\nabla w}+\norm[H^{1/2}]{h}
 \geq cK^{1/4}-C.
$$
Since $\norm[W^{1,1}]{w}\geq\norm[L^1]{\nabla w}$, taking the infimum over all decompositions yields
\begin{equation}
 \norm[H^{1/2}(\T^2)+W^{1,1}(\T^2)]{\Phi_K}
 \geq cK^{1/4}-C.
 \label{eq:left-lower-2d}
\end{equation}

\subsection{The ambient upper bound}

It remains to estimate the decomposition in \eqref{eq:ambient-decomp-preview}.  Orthogonality gives $\norm[L^2_x]{b_K}^2=1/2$; hence Cauchy--Schwarz and the normalization of $\rho_{\varepsilon_K}$ imply
\begin{equation}
 \norm[L^1(\T^2)]{a_K}
 =\frac1{2\pi}\norm[L^1(\T)]{b_K}
 \leq\frac1{2\pi}\norm[L^2(\T)]{b_K}\leq C.
 \label{eq:a-l1}
\end{equation}
For a frequency $(n,\ell)$ with $n\neq0$, the symbol \eqref{eq:hodge-symbol} gives
\begin{equation}
 \abs{Q(n,\ell)e_t}^2=\frac{n^2}{n^2+\ell^2}.
 \label{eq:normal-Q}
\end{equation}
Moreover, $|\widehat\rho_{\varepsilon_K}(\ell)|\leq1$.  Absorbing the fixed factor $(2\pi)^{-1}$ in the constant, we obtain
\begin{equation}
 \norm[H^{-1/2}]{Qa_K}^2
 \leq C\sum_{n\ne0}\abs{\widehat b_K(n)}^2
    \sum_{\ell\in\Z}
      \frac{n^2}{(n^2+\ell^2)^{3/2}}.
 \label{eq:Q-bound-start}
\end{equation}
The inner sum is bounded uniformly for integers $|n|\geq1$: the term $\ell=0$ is at most one, the terms $1\leq|\ell|\leq|n|$ contribute $O(1)$, and the tail is bounded by $Cn^2\sum_{\ell>|n|}\ell^{-3}=O(1)$.  Parseval gives
\begin{equation}
 \norm[H^{-1/2}]{Qa_K}\leq C.
 \label{eq:Q-bound}
\end{equation}
Using \eqref{eq:ambient-decomp-preview} in the sum norm, together with \eqref{eq:a-l1} and \eqref{eq:Q-bound}, gives
\begin{equation}
 \norm[H^{-1/2}(\T^2)+L^1(\T^2)]{\nabla\Phi_K}
 \leq2\norm[L^1]{a_K}+2\norm[H^{-1/2}]{Qa_K}\leq C.
 \label{eq:right-upper-2d}
\end{equation}
Equations \eqref{eq:left-lower-2d} and \eqref{eq:right-upper-2d} prove Theorem \ref{thm:main} in dimension two.

\section{Higher dimensions}
\label{sec:higher-dimensions}

For $N>2$, extend the two-dimensional example by
$$
 \widetilde\Phi_K(x_1,x'',x_N)=\Phi_K(x_1,x_N),
 \qquad x''\in\T^{N-2}.
$$
Also set
$$
 \widetilde a_K(x_1,x'',x_N)
 =\frac1{2\pi}b_K(x_1)e_N\rho_{\varepsilon_K}(x_N).
$$
Under normalized Haar measure, the $L^1$ and Sobolev norms of fields independent of $x''$ agree with their two-dimensional norms. On frequencies $(n,0,\ldots,0,\ell)$, the $N$-dimensional Hodge symbol restricted to the $(x_1,x_N)$ plane is the two-dimensional symbol. Consequently,
$$
 P_N\widetilde a_K=\widetilde{P_2a_K},
 \qquad
 \nabla\widetilde\Phi_K=2P_N\widetilde a_K,
$$
where the tilde inserts zero components in the $x''$ directions. Hence \eqref{eq:right-upper-2d} remains valid.

For the lower bound, let
$$
 \mathcal AF(x_1,x_N)=\int_{\T^{N-2}}
 F(x_1,x'',x_N)\,dm_{N-2}(x'').
$$
Jensen's inequality and Fubini give
$$
 \norm[W^{1,1}(\T^2)]{\mathcal AW}
 \leq\norm[W^{1,1}(\T^N)]{W},
$$
while projection onto the Fourier modes with zero $x''$ frequency gives
$$
 \norm[H^{1/2}(\T^2)]{\mathcal AH}
 \leq\norm[H^{1/2}(\T^N)]{H}.
$$
Every decomposition $\widetilde\Phi_K=W+H$ therefore averages to a two-dimensional decomposition $\Phi_K=\mathcal AW+\mathcal AH$ of no greater cost.  Taking infima and applying \eqref{eq:left-lower-2d} proves \eqref{eq:lhs-main} for every $N\geq2$, completing the proof of Theorem \ref{thm:main}.
The same two-dimensional constants work for every $N\geq2$.

\section*{Declaration of generative AI and AI-assisted technologies in the manuscript preparation process}

During the preparation of this work, the author used OpenAI ChatGPT, including the GPT-5.6 Sol model, to assist with exploratory calculations, proof organization, consistency checks, literature searches, and language editing. The author independently verified the mathematical arguments and computations, reviewed and edited the AI-assisted material, and takes full responsibility for the content of the article.


\begin{thebibliography}{99}

\bibitem{Arsenio2026} D. Ars\'enio,
\newblock The trace-free Beurling--Ahlfors transform and the Bourgain--Brezis problem for Hodge systems,
\newblock \emph{arXiv:2608.04237v1} (2026).

\bibitem{BerghLofstrom1976} J. Bergh and J. L\"ofstr\"om,
\newblock \emph{Interpolation Spaces: An Introduction},
\newblock Grundlehren der mathematischen Wissenschaften, vol.~223, Springer, Berlin, 1976.

\bibitem{BB03} J. Bourgain and H. Brezis,
\newblock On the equation $\operatorname{div}Y=f$ and application to control of phases,
\newblock \emph{J. Amer. Math. Soc.} \textbf{16} (2003), no.~2, 393--426,
\newblock doi:10.1090/S0894-0347-02-00411-3.

\bibitem{BB04} J. Bourgain and H. Brezis,
\newblock New estimates for the Laplacian, the div--curl, and related Hodge systems,
\newblock \emph{C. R. Math. Acad. Sci. Paris} \textbf{338} (2004), no.~7, 539--543,
\newblock doi:10.1016/j.crma.2003.12.031.

\bibitem{BB07} J. Bourgain and H. Brezis,
\newblock New estimates for elliptic equations and Hodge type systems,
\newblock \emph{J. Eur. Math. Soc.} \textbf{9} (2007), no.~2, 277--315,
\newblock doi:10.4171/JEMS/80.

\bibitem{BBM00} J. Bourgain, H. Brezis and P. Mironescu,
\newblock Lifting in Sobolev spaces,
\newblock \emph{J. Anal. Math.} \textbf{80} (2000), 37--86,
\newblock doi:10.1007/BF02791533.

\bibitem{BBM04} J. Bourgain, H. Brezis and P. Mironescu,
\newblock $H^{1/2}$ maps with values into the circle: minimal connections, lifting, and the Ginzburg--Landau equation,
\newblock \emph{Publ. Math. Inst. Hautes \'Etudes Sci.} \textbf{99} (2004), 1--115,
\newblock doi:10.1007/s10240-004-0019-5.

\bibitem{BMR13} P. Bousquet, P. Mironescu and E. Russ,
\newblock A limiting case for the divergence equation,
\newblock \emph{Math. Z.} \textbf{274} (2013), no.~1--2, 427--460,
\newblock doi:10.1007/s00209-012-1077-x.

\bibitem{BRWY19} P. Bousquet, E. Russ, Y. Wang and P.-L. Yung,
\newblock Approximation in higher-order Sobolev spaces and Hodge systems,
\newblock \emph{J. Funct. Anal.} \textbf{276} (2019), no.~5, 1430--1478,
\newblock doi:10.1016/j.jfa.2018.08.003.

\bibitem{Brezis2023} H. Brezis,
\newblock Some of my favorite open problems,
\newblock \emph{Rend. Lincei Mat. Appl.} \textbf{34} (2023), no.~2, 307--335,
\newblock doi:10.4171/RLM/1008.

\bibitem{BM21} H. Brezis and P. Mironescu,
\newblock \emph{Sobolev Maps to the Circle: From the Perspective of Analysis, Geometry, and Topology},
\newblock Progress in Nonlinear Differential Equations and Their Applications, vol.~96, Birkh\"auser/Springer, New York, 2021,
\newblock doi:10.1007/978-1-0716-1512-6.

\bibitem{Curca2025} E. Curc\u{a},
\newblock Pathological sums of Sobolev spaces and some counterexamples,
\newblock \emph{Math. Z.} \textbf{311} (2025), article 8,
\newblock doi:10.1007/s00209-025-03802-x.

\bibitem{Curca2026} E. Curc\u{a},
\newblock Bourgain--Brezis spaces obtained by real interpolation,
\newblock \emph{arXiv:2604.23158v1} (2026).

\bibitem{DLRW2021} F. Da Lio, T. Rivi\`ere and J. Wettstein,
\newblock Bergman--Bourgain--Brezis-type inequality,
\newblock \emph{J. Funct. Anal.} \textbf{281} (2021), no.~9, article 109201,
\newblock doi:10.1016/j.jfa.2021.109201.

\bibitem{DaiEtAl2026} X. Dai, W. Deng, Y. Shi, T. Wu and Y. Yang,
\newblock Bounded representatives in critical Sobolev--Hodge spaces,
\newblock \emph{arXiv:2608.07845v1} (2026).

\bibitem{DNPV2012} E. Di Nezza, G. Palatucci and E. Valdinoci,
\newblock Hitchhiker's guide to the fractional Sobolev spaces,
\newblock \emph{Bull. Sci. Math.} \textbf{136} (2012), no.~5, 521--573,
\newblock doi:10.1016/j.bulsci.2011.12.004.

\bibitem{Katznelson2004} Y. Katznelson,
\newblock \emph{An Introduction to Harmonic Analysis}, third ed.,
\newblock Cambridge Mathematical Library, Cambridge University Press, Cambridge, 2004.

\bibitem{LionsMagenes1972} J.-L. Lions and E. Magenes,
\newblock \emph{Non-Homogeneous Boundary Value Problems and Applications}, vol.~I,
\newblock Grundlehren der mathematischen Wissenschaften, vol.~181, Springer-Verlag, Berlin, 1972,
\newblock doi:10.1007/978-3-642-65161-8.

\bibitem{Mazya2007} V. Maz'ya,
\newblock Bourgain--Brezis type inequality with explicit constants,
\newblock in L. De Carli and M. Milman (eds.), \emph{Interpolation Theory and Applications},
\newblock Contemp. Math., vol.~445, Amer. Math. Soc., Providence, RI, 2007, 247--252,
\newblock doi:10.1090/conm/445/08605.

\bibitem{Mironescu2010} P. Mironescu,
\newblock On some inequalities of Bourgain, Brezis, Maz'ya, and Shaposhnikova related to $L^1$ vector fields,
\newblock \emph{C. R. Math. Acad. Sci. Paris} \textbf{348} (2010), no.~9--10, 513--515,
\newblock doi:10.1016/j.crma.2010.03.019.

\bibitem{VanSchaftingen2014} J. Van Schaftingen,
\newblock Limiting Bourgain--Brezis estimates for systems of linear differential equations: theme and variations,
\newblock \emph{J. Fixed Point Theory Appl.} \textbf{15} (2014), no.~2, 273--297,
\newblock doi:10.1007/s11784-014-0177-0.

\bibitem{Zygmund2002} A. Zygmund,
\newblock \emph{Trigonometric Series}, vols.~I and II, third ed.,
\newblock Cambridge Mathematical Library, Cambridge University Press, Cambridge, 2002.

\end{thebibliography}
\end{document}